\documentclass[12pt]{amsart}
\usepackage{amscd,amsthm,amssymb,amsfonts,amsmath,epsfig,epic,bbm,url}

\theoremstyle{plain}
\newtheorem{theorem}{Theorem}
\newtheorem{lemma}[theorem]{Lemma}

\theoremstyle{definition}
\newtheorem{definition}[theorem]{Definition}

\theoremstyle{remark}
\newtheorem*{remark}{Remark}
\newtheorem*{acknowledgments}{Acknowledgments}

\newcommand{\Z}{\mathbb Z}    % Integers
\newcommand{\R}{\mathbb R}    % Reals
\newcommand{\PP}{\mathbb P}   % Projective space
\newcommand{\F}{\mathcal F} % homology sheaves
\newcommand{\<}{\langle}   %\< is not defined yet.
\renewcommand{\>}{\rangle} %\> is already defined.

\renewcommand{\Im}{\operatorname{Im}}

\newcommand{\ignore}[1]{\relax}
\newcommand{\dd}{\partial}

\newcommand{\Hom}{\operatorname{Hom}}
\newcommand{\rk}{\operatorname{rk}}
\newcommand{\OS}{\mathcal I}  %Orlik-Solomon ideal

\begin{document}

\title[The Orlik-Solomon Algebra and the Bergman Fan]{The Orlik-Solomon Algebra and the Bergman Fan of a Matroid}

\author{Ilia Zharkov}
\address{Kansas State University, 138 Cardwell Hall, Manhattan, KS 66506}
\email{zharkov@math.ksu.edu}

\begin{abstract}
Given a matroid $M$ one can define its Orlik-Solomon algebra $OS(M)$ and the Bergman fan $\Sigma_0(M)$. On the other hand to any rational polyhedral fan $\Sigma$ one can associate its tropical homology and cohomology groups  $\F_\bullet(\Sigma)$, $\F^\bullet (\Sigma)$. We show that the projective Orlik-Solomon algebra $OS_0(M)$ is canonically isomorphic to $\F^\bullet (\Sigma_0(M))$. In the realizable case this provides a geometric interpretation of homology of the complement of the correspoding hyperplane arrangement in $\PP^n$.

%This note was meant to serve as an appendix to the tropical homology paper \cite{IKMZ}. But since a shorter geometric argument will appear in the realizable case (all that is needed there) in \cite{IKMZ} the algebraic proof given here makes now more sense as an independent article.
%
\end{abstract}
\maketitle

\footnote{The research is partially supported by the NSF FRG grant DMS-0854989.

2010 AMS subject classification 14T05, 52B40}

\section{Notations and Statements}

\subsection{Tropical homology and cohomology}
Let $\Sigma=\bigcup \sigma\subset \R^N=\Z^N\otimes \R$ be an integral polyhedral fan.  For each cone $\sigma\subset \Sigma$ we denote by $\<\sigma\>_\Z$ the integral lattice in the vector subspace linearly spanned by $\sigma$. 
\begin{definition}\cite{IKMZ}
The {\em homology} group $\F_k(\Sigma)$ is the subgroup of $ \wedge^k\Z^N$ generated by the elements $v_1\wedge\dots \wedge v_k$, where all $v_1,\dots,v_k \in \<\sigma\>_\Z$ for some cone $\sigma\in\Sigma$. The {\em cohomology} is the dual group $\F^k(\Sigma):=\Hom(\F_k(\Sigma),\Z)$, which is the quotient of $\wedge^\bullet(\Z^N)^*$  by $(\F_k)^\perp$.
\end{definition}

%\begin{remark} For the isomorphism $\wedge^\bullet(\Z^N)^*$ with $(\wedge^\bullet \Z^N)^*$ to be canonical we use the convention that $\wedge^\bullet \Z^N $ is defined as the subgroup of $\otimes^\bullet \Z^N$, and $\wedge^\bullet(\Z^N)^*$ is the quotient of $\otimes^\bullet(\Z^N)^*$.
%\end{remark}

\begin{lemma}
The wedge product on  $\wedge^\bullet(\Z^N)^*$ descends to $\F^\bullet$, that is, $\F^\bullet$ is endowed with a natural algebra structure over $\Z$.
\end{lemma}
\begin{proof}
We just need to show that the subgroup of $\wedge^\bullet(\Z^N)^*$ annihilating $\F_\bullet$ forms an ideal. Let $f\in (\F_k)^\perp$, then for any $\alpha\in (\Z^N)^*$ and any collection $v_0,v_1,\dots,v_k\in \<\sigma\>_\Z$ we have 
$$(\alpha\wedge f)(v_0\wedge v_1\wedge\dots \wedge v_k)=\sum_{i=0}^k (-1)^i \alpha(v_i) f(v_0\wedge \dots \hat v_i \dots \wedge v_k ),
$$
which vanishes since any $k$-subset of $v_0,v_1,\dots,v_k$ is also in $\<\sigma\>_\Z$.
Hence $\alpha \wedge f$ is in $(\F_{k+1})^\perp$.
\end{proof}

\subsection{The Bergman fan}
 Let $M$ be a loopless matroid of rank $n$ on the set 
$\{0,\dots,N\}$. Let $V$ be the rank $N+1$ free abelian group generated by elements $e_0,\dots,e_N$. Consider the simplicial fan $\Sigma(M) \subset V_\R$ built on the lattice of flats of $M$.
Namely, the rays of $\Sigma$ are along the vectors $e_J:=e_{j_1} +\dots +e_{j_k}$ for each flat $J=\{{j_1}, \dots ,{j_k}\}$. The $k$ dimensional cones of $\Sigma(M)$ are spanned by the $k$-tuples of  rays indexed by flags of flats of length $k$. We will also use the notation 
$$E_I:=e_{i_1} \wedge \dots \wedge e_{i_k}
$$ 
for any subset $\{i_1,\dots,i_k\}\subset M$. Note the distinction between $E_I$ and $e_I$.
We reserve letter $J$ to denote flats in $M$, while $I$ will be used for general subsets of $M$.

The Bergman fan is the quotient fan $\Sigma_0(M)$ of $\Sigma(M)$ at the ray $e_M$. Namely, it is defined like above by the lattice of proper flats of $M$ in the quotient lattice $V_0=V / \<e_0+\dots+e_N\>$. 

Sturmfels \cite{Sturmfels} noticed the importance of the Bergman fan in tropical geometry
where it represents a linear space. Later
Ardila and Klivans \cite{AK} studied its combinatorics and showed, among other things, that $\Sigma_0(M)$ is indeed a balanced fan of degree 1. 
%Here we adopted the definition from \cite{KP}.

\subsection{The Orlik-Solomon algebra} To the same matroid $M$ one can associate its Orlik-Solomon algebra $OS^\bullet(M)$ over $\Z$ defined below. In case $M$ is realizable by a hyperplane arrangement in $\PP^{n-1}$,  the projectivized version
$OS_0^\bullet(M)$ of this algebra calculates the cohomology of the complement of this arrangement. (See \cite{OT} for more details).

Let $W$ be the rank $N+1$ free abelian group generated by elements $f_0,\dots,f_N$. Then $OS^\bullet(M):=\wedge^\bullet W/\OS^\bullet$, where the Orlik-Solomon ideal $\OS$ is generated by the elements  
$$\partial(f_{i_0}\wedge f_{i_1}\wedge \dots \wedge f_{i_k}):=\sum_{s=0}^k (-1)^s f_{i_0}\wedge \dots \hat f_{i_s} \dots \wedge f_{i_k}, 
$$
for all dependent subsets $I=\{{i_0},{i_1}, \dots ,{i_k}\}$. We will use the notation 
$$F_I:=f_{i_0}\wedge f_{i_1}\wedge \dots \wedge f_{i_k}.
$$
The sign of $F_I$ depends on the order of $I$, so we assume that all subsets of $M$ are ordered.

The projective Orlik-Solomon algebra $OS_0^\bullet(M)$ is defined as follows. Let $W_0$ be the subgroup of $W$ generated by all differences $f_i-f_j$. Then we set $OS_0^\bullet(M):=\wedge^\bullet W_0/\OS^\bullet_0$, where $\OS_0$ is the restriction of $\OS$ to the subalgebra $\wedge^\bullet W_0 \subset \wedge^\bullet W$.

\begin{theorem}\label{lemma:OS}
Consider $V$ and  $W$ above as dual groups with the dual bases  $\{e_0,\dots,e_N\}$ and $\{f_0,\dots,f_N\}$. Then $\F_k (\Sigma(M))^\perp=\OS^k$.
\end{theorem}

An important corollary of this theorem is our main result:

\begin{theorem}\label{thm:OS}
There is a canonical isomorphism $\F^\bullet(\Sigma_0(M)) \cong OS_0^\bullet(M)$ of  graded algebras.
\end{theorem}

\section{Two examples}
Example 1: Matroid $M_1$ on 4 elements of rank 2 represented by 4 lines in $\PP^2$ (see Fig. \ref{fig:1}).
 \begin{figure}[htb]
       \includegraphics[height=30mm]{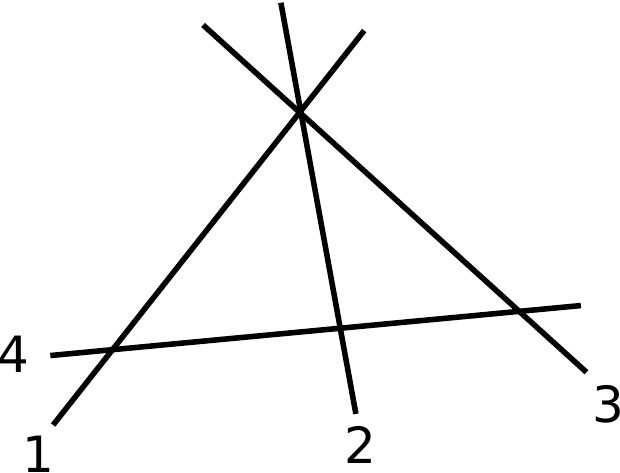}
   \qquad     
     \includegraphics[height=40mm]{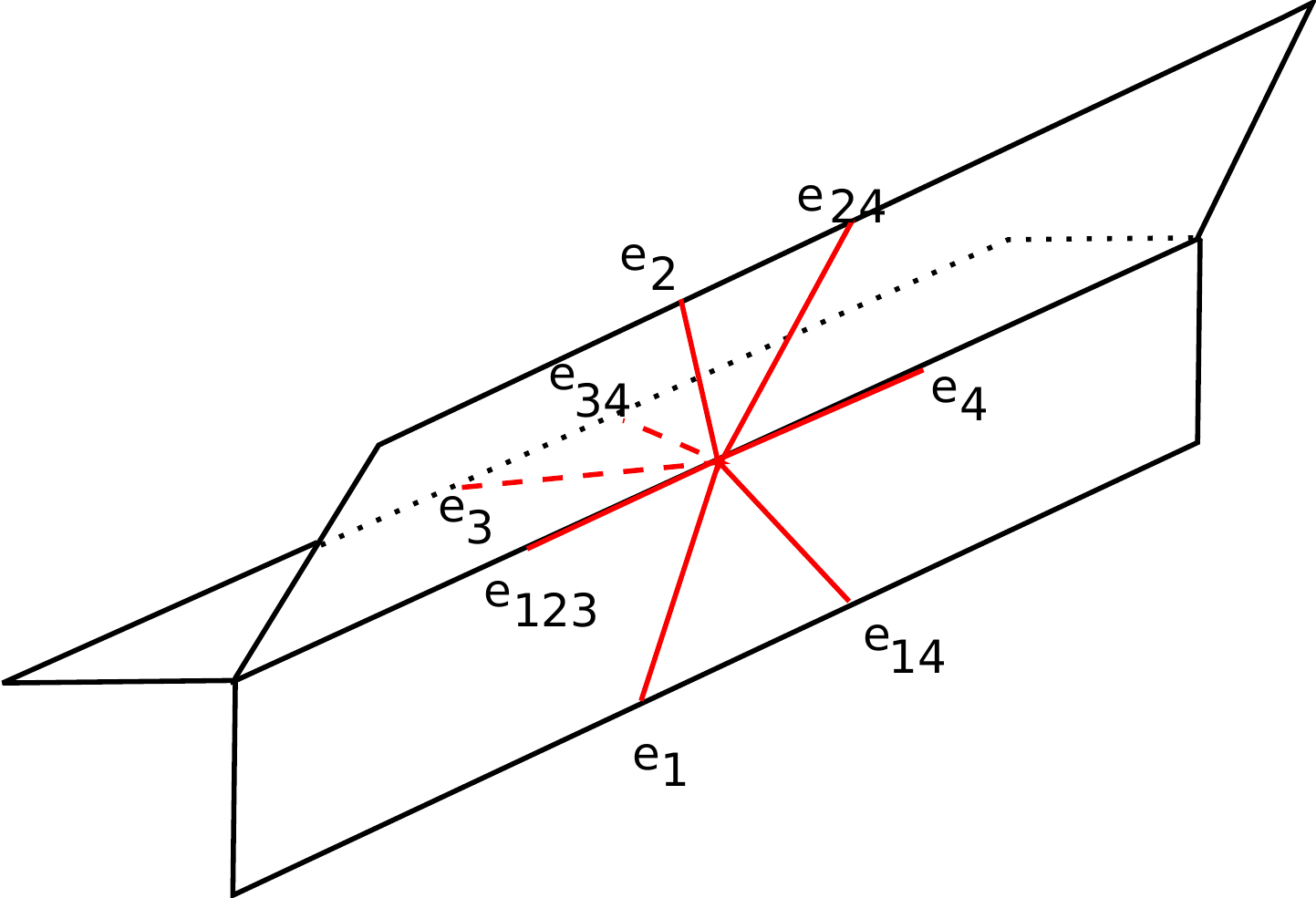} 
        \caption{Matroid $M_1$ and its Bergman fan.}\label{fig:1}
   \end{figure}
  
   The flats are: 
\begin{gather*}
1234\\
123 \quad 14 \quad 24 \quad 34\\
 1 \quad  2 \quad  3 \quad 4
\end{gather*}
and the only circuit is $123$. Thus the 
Orlik-Solomon ideal is generated by $\partial F_{123}=F_{12}+F_{23}+F_{31}$. On the other hand $\F_2 (\Sigma_0 (M_1)) \cong \Z^2$ is generated by $E_{i4}=e_i\wedge e_4, \ i=1,2,3$. It is clear that $F_{12}+F_{23}+F_{31}$ is the only (upto scalar) orthogonal bivector to all $E_{i4}=e_i\wedge e_4, \ i=1,2,3$.

Example 2: Matroid $M_2$ on 6 elements of rank 2 represented by 6 lines in $\PP^2$ (see Fig. \ref{fig:2}). It is isomorphic to the graphical matroid for the complete graph $K_4$.

 \begin{figure}[htb]
       \includegraphics[height=30mm]{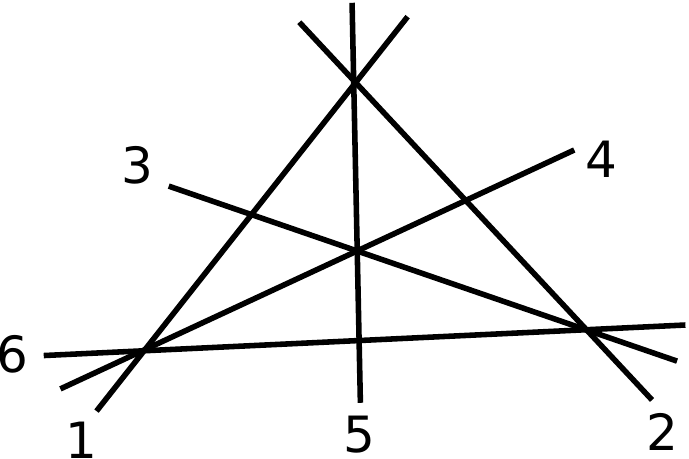}
   \qquad     
     \includegraphics[height=40mm]{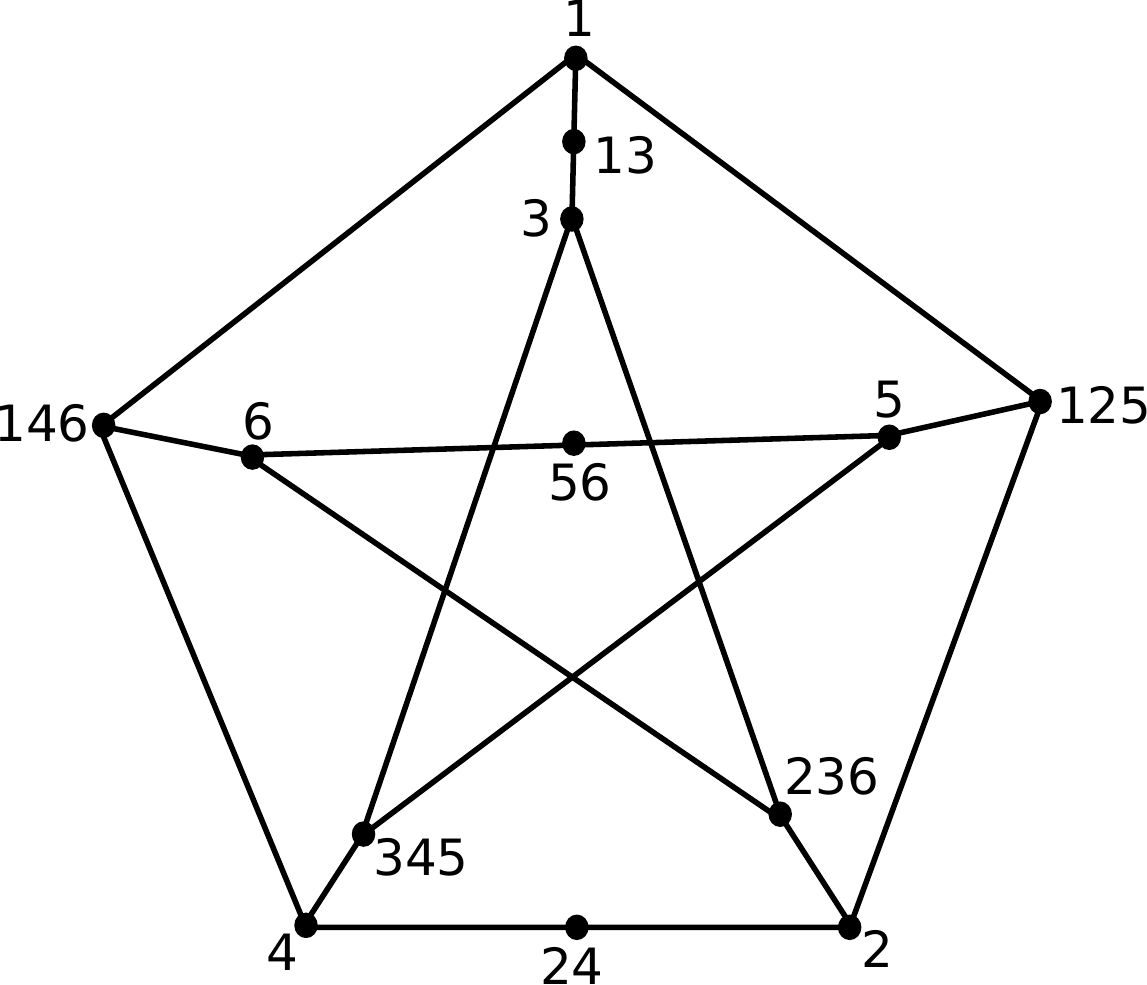} 
        \caption{Matroid $M_2$.  Its Bergman fan in $\R^5$ is combinatorially the cone over the (subdivided) Petersen graph.}\label{fig:2}
   \end{figure}

The flats are: 
\begin{gather*}
123456\\
125 \quad 13 \quad 146 \quad 24 \quad 236 \quad 56 \quad 345\\
 1 \quad  2 \quad  3 \quad 4 \quad 5 \quad 6
\end{gather*}
and the circuits of rank 2 are $125$, $146$, $236$ and $345$.  There are also 3 circuits of rank 3: $1234$, $1356$ and $2456$. Thus the 
Orlik-Solomon ideal  in degree 2 is generated by 
\begin{gather*}\label{eq:ideal}
\partial F_{125}=F_{12}+F_{25}+F_{51}\\
\partial F_{146}=F_{14}+F_{46}+F_{61}\\
\partial F_{236}=F_{23}+F_{36}+F_{62}\\
\partial F_{345}=F_{34}+F_{45}+F_{53}
\end{gather*}
On the other hand $\F_2 (\Sigma_0 (M_2)) \cong \Z^6$ is generated by the 15 bivectors (one for each 2-dimensional cone) with 10 relations among them (one for each ray of $\Sigma_0(M_2))$). And there is a relation among the relations (the sum is tautologically 0 in $\Lambda^2 \Z^5$).
 One easily sees that all 15 bivectors are orthogonal to the Orlik-Solomon elements in $\OS_0^2(M_2)$ above. Counting dimensions we can conclude that  $\OS_0^2(M_2)\subset \Lambda^2 (\Z^{5})^*$ is the orthogonal subgoup to $\F_2 (\Sigma_0 (M_2))$.

\section{Proofs of Theorems \ref{lemma:OS} and  \ref{thm:OS}}
For a flat $J\subset M$ we consider the restricted groups $\F_k(J):= \F_k(\Sigma(J))$ as subgroups of $\wedge^\bullet V$ under the natural embedding $\wedge^\bullet (\Z \<e_j, {j\in J}\>)\subset \wedge^\bullet V$.
We also consider the restricted Orlik-Solomon algebra $OS^\bullet(J):=\wedge^\bullet W/\OS^\bullet (J)$ by defining the ideal $\OS^\bullet(J)\subset \wedge^\bullet W$ to be generated by the $\dd F_I$ with dependent $I\subset J$, and by the $f_i, i \not\in J$.

\begin{lemma}\label{lemma:Fk}
$\F_k(M)=\Z\< \F_k(J)\>_{\rk(J)=k}$.
\end{lemma}
\begin{proof}
Let $J \subset J''$ be two flats whose ranks differ by 2 or more. Let $J'_1, \dots J'_s$ be the set of flats between $J$ and $J''$ of rank exactly one larger than the rank of $J$. Then the sets $J, J'_1\setminus J, \dots, J'_s\setminus J$ give a partition of $J''$. Hence 
$$ e_{J} \wedge e_{J''}=\sum_{i=1}^s e_J \wedge e_{J'_i}.$$
By induction, for any $k$-flag of flats $J_1\subset \dots \subset J_k$ the element $e_{J_1}\wedge \dots \wedge e_{J_k}$ can be rewritten as a sum
$$e_{J_1}\wedge \dots \wedge e_{J_k}= \sum e_{J''_1}\wedge \dots e_{J''_k},$$
where all flags $J''_1\subset \dots \subset J''_k$ consist of flats of ranks $1,\dots, k$, respectively. 
\end{proof}

\begin{lemma}\label{lemma:Leibnitz}
As an abelian group $\OS^k = \Z\<F_{I'}, \dd F_{I''}\>$, where $I'$ and $I''$ run over dependent sets in $M$ of size $k$ and $k+1$, respectively. In particular, in the top degree $\OS^n = \Z\<F_I , \Im \{\dd : \wedge^{n+1} W \to \wedge^n W\}\>=\Z\<F_I , \ker \{\dd : \wedge^{n} W \to \wedge^{n-1} W\}\>$, where $I$ runs over dependent sets of size $n$.
\end{lemma}

\begin{proof}
The statement follows from the Leibnitz rule:
$$\partial F_{I'}\wedge F_{I''}=\partial ( F_{I'}\wedge F_{I''}) \pm F_{I'}\wedge \partial F_{I''}.
$$
For the top degree note that every $(n+1)$-set is dependent and $(\wedge^\bullet W, \dd)$ is an acyclic complex, that is $ \Im \dd=\ker \dd $.
\end{proof}

\begin{remark}
For the second subset of generators it is enough to take $\dd F_I$ with rank of $I$ exactly $k$, since $\dd F_I$ with $I$ of smaller ranks are already included in the first subset of generators.
\end{remark}

\begin{remark} 
Note that the projective Orlik-SOlomon ideal $\OS_0$ is generated as an abelian group just by the $\partial F_I$ for dependent $I$. Indeed, note that  $\wedge^\bullet W_0 = \Im \{\dd : \wedge^\bullet W \to \wedge^\bullet W\} = \ker \{\partial: \wedge^\bullet W \to \wedge^\bullet W\}$. But from the Lemma \ref{lemma:Leibnitz} if $\alpha\in \OS$ we can write $\alpha=\sum F_{I'} + \sum \dd F_{I''}$, with all $I',I''$ dependent sets. On the other hand $\dd \alpha=\dd (\sum F_{I'}) =0$ means $\sum F_{I'} = \dd \sum F_{\hat I}$, where every $\hat I$ is an extension by one element of some dependent $I$, and hence is also dependent.
\end{remark}

\begin{lemma}\label{lemma:OSk}
$\OS^k =\cap_{\rk(J)=k} \OS^k(J)$.
\end{lemma}

\begin{proof}
First we argue that for any rank $k$ flat $J$ we have  $\OS^k \subset  \OS^k(J)$. Indeed, according to  Lemma \ref{lemma:Leibnitz} and the first remark after it we just have to show that $\partial F_{\hat I} \in \OS^k(J)$ for $\hat I$ of size $k+1$ and rank $k$. If $\hat I \subset J$ or $|\hat I \setminus J|\ge 2$, we are done. Otherwise, say $|\hat I \setminus J|=\{s\}$. Then
$$\partial F_{\hat I}= f_s\wedge(\dots) \pm F_{\hat I\setminus s}.
$$
But $\hat I\setminus s \subset J$ must have rank $k-1$ (or, otherwise $\hat I \subset J$), hence is dependent. Thus $F_{\hat I\setminus s}$ is in $\OS^k(J)$, and so is $\partial F_{\hat I}$.
Consequently,  $\OS^k \subset \cap_{\rk(J)=k} \OS^k(J)$. 

To show the converse we notice that each $I$ is contained in a unique flat of the same rank (the matroidal closure of $I$). We group the terms in an element $\alpha=\sum F_I \subset\wedge^{k} W$ by their flats: 
$$\alpha=\alpha_{< k}+\sum_{\rk(J')=k} \alpha_{J'}$$
where $\alpha_{< k}$ contains terms $F_I$ with dependent $I$. 

Now if $\alpha\in \OS^k(J)$ for some rank $k$ flat $J$, then in the above decomposition $\alpha_{< k} \in \OS^k \subset  \OS^k(J)$. Also all $\alpha_{J'}$ with $J'\ne J$, are in $\OS^k(J)$, and hence so is $\alpha_J$. But all terms $F_I$ in $\OS^k(J)$ with independent $I\subset J$ have to come from $\dd F_{\hat I}$ for some dependent $\hat I\subset J$. Thus $\alpha_J \in \OS^k$. Taking the intersection over all  $k$-flats completes the proof.
\end{proof}

\begin{proof}[Proof of Theorem \ref{lemma:OS}]
Taking the intersection in Lemma \ref{lemma:OSk} is orthogonal to taking the sum in Lemma \ref{lemma:Fk}. Thus it is enough to prove the statement in the top degree for any matroid. By dualizing  the top degree part of Lemma \ref{lemma:Leibnitz}
it suffices then to show that $\F_n = \<E_I\>\cap \Im (\dd^*)$, where the $I$ run over independent $n$-sets of $M$ and $\dd^*: \wedge^{n-1} V \to \wedge^n V, \  \dd^*(E_I)=e_M\wedge E_I$ is the adjoint operator to $\dd:\wedge^{n} W \to \wedge^{n-1}W$.

For any complete flag of flats $J_1\subset\dots \subset J_{n-1}\subset M$ the polyvector
$$e_{J_1}\wedge \dots \wedge e_M=e_{J_1}\wedge e_{J_2\setminus J_1}\wedge \dots \wedge e_{M\setminus J_{n-1}}=\sum E_I
$$ 
contains only terms with independent $I$. Thus $\F_n\subset \<E_I\>\cap \Im (\dd^*)$.
 We will show the converse by induction on the rank of $M$. For rank 1 matroids both spaces are $\Z\<e_M\>$ and there is nothing to prove. 

Suppose now $\alpha=e_M\wedge \beta = \sum E_{\hat I}$, with all $\hat I$ independent. We may choose a representation for $\beta=\sum E_I$ with all $I$ independent subsets as follows. Substituting, say, $e_0=-\sum_{i=1}^N e_i$ mod $e_M$ into $\beta$ we will have
$$\alpha=e_0 \wedge \beta+(\text{terms with no }e_0)$$
and any term $E_I$ with dependent $I$ in $\beta$ will result in $E_{0\cup I}$ in $\alpha$ with dependent $\hat I= 0\cup I$ which cannot happen.

Let $J_1,\dots, J_r$ be all rank $(n-1)$ flats in $M$. We again group the terms in $\beta$ by the respective flats $\beta= \beta_{J_1}+\dots+\beta_{J_r}$. Then writing 
$$\alpha= (e_{J_1} \wedge \beta_{J_1}+\dots+e_{J_r} \wedge \beta_{J_r})+
(e_{M \setminus J_1} \wedge \beta_{J_1}+\dots+e_{M\setminus J_r} \wedge \beta_{J_r})
$$
we note that both $\alpha$ and the second summand contain terms $E_I$ only with independent $I$. On the other hand, each $e_{J_k} \wedge \beta_{J_k}$ contains terms of rank $n-1$, and there no cancellations possible among different $k$. Thus $e_{J_k} \wedge \beta_{J_k}=0$. By exactness of the $(e_{J_k}\wedge)$-operator we can write each $\beta_{J_k}$ as $e_{J_k}\wedge \gamma_{J_k}$ and use the induction assumption.
\end{proof}

\begin{proof}[Proof of Theorem \ref{thm:OS}]
With the choice of the dual bases for $V$ and $W$ the restriction to $W_0$ in $W$ is exactly dual to the quotient by $e_M$ in $V$, and the duality extends to the exterior algebras. On the other hand, the identification in Theorem \ref{lemma:OS} clearly extends to the level of graded ideals $\F_\bullet (\Sigma(M))^\perp=\OS^\bullet$, as well as to their restrictions to $\wedge^\bullet W_0$. 
\end{proof}

\begin{acknowledgments}
This note was meant to serve as an appendix to our (long overdue) joint project on tropical homology \cite{IKMZ} with Ilia Itenberg, Ludmil Katzarkov and Grisha Mikhalkin. But since a shorter geometric argument will appear in the realizable case  in \cite{IKMZ} (all that is needed there) the algebraic proof given here makes now more sense as an independent article.
Of course, numerous discussions with all three coauthors were crucial for the formulation and the proof of the main statement. I am very grateful to them for their permission to publish the result as a separate paper. 
I would also like to thank Federico Ardila for several very helpful conversations. He is primerily responsible for my current appreciation of matroids. Finally I should mention that another inductive proof of Theorem  \ref{thm:OS} was recently given by Kristin Shaw in her thesis \cite{Shaw} using tropical modifications.
\end{acknowledgments}

\end{document}